\documentclass[11pt,a4paper]{amsart}
\pdfoutput=1
\usepackage{times}
\usepackage{amssymb}

\usepackage{ifthen}
\usepackage{cite}
\usepackage{graphicx}

\makeatletter

\newcommand{\gl}{\lambda}

\newcommand{\gs}{\sigma}

\newcommand{\gO}{\Omega}

\newcommand{\cA}{\mathcal{A}}
\newcommand{\cB}{\mathcal{B}}

\newcommand{\N}{\mathbb{N}}

\newcommand{\R}{\mathbb{R}}

\newtheorem*{theorem*}{Theorem}
\newtheorem{lemma}{Lemma}
\newtheorem{corollary}[lemma]{Corollary}
\newtheorem{definition}[lemma]{Definition}

\theoremstyle{remark}

\newenvironment{enum_a}
    {\begin{enumerate}}
    {\end{enumerate}}

\newif\if@golden  \@goldentrue
\newcommand{\f@ctor}{1}
\newlength{\aiv@width}  
\newlength{\aiv@height} 
\newlength{\tmp@width}  
\newlength{\tmp@height} 
\if@twocolumn\if@golden
  \else\fi
\else\if@golden\ifcase\@ptsize\relax\or
\or\fi
  \else\ifcase\@ptsize\relax\or
\or\fi\fi
\if@twocolumn\if@golden\ifcase\@ptsize\relax\renewcommand{\f@ctor}{53}
  \or\renewcommand{\f@ctor}{46}\or\renewcommand{\f@ctor}{43}\fi
  \else\ifcase\@ptsize\relax\renewcommand{\f@ctor}{51}\or
  \renewcommand{\f@ctor}{45}\or\renewcommand{\f@ctor}{42}\fi\fi
\else\if@golden\ifcase\@ptsize\relax \renewcommand{\f@ctor}{46}
  \or\renewcommand{\f@ctor}{43}\or\renewcommand{\f@ctor}{43}\fi
  \else\ifcase\@ptsize\relax\renewcommand{\f@ctor}{43}
  \or\renewcommand{\f@ctor}{40}\or\renewcommand{\f@ctor}{40}\fi\fi\fi
\multiply\textheight by \f@ctor
\let\comp\circ
\newcommand{\Ra}{\space$\Rightarrow$\space}

\newcommand{\Coe}{\ensuremath C_0(E)}

\makeatother
\allowdisplaybreaks[2]

\date{February 18, 2011}

\title[Feller Semigroups and Resolvents]{%
A Note on Feller Semigroups and Resolvents}

\author[V.~Kostrykin]{Vadim Kostrykin}
\address{Vadim Kostrykin\newline
Institut f\"ur Mathematik\newline
Johannes Gutenberg--Universit\"at\newline
D--55099 Mainz, Germany}
\email{kostrykin@mathematik.uni-mainz.de}

\author[J.~Potthoff]{J\"urgen Potthoff}
\address{J\"urgen Potthoff\newline
Institut f\"ur Mathematik, Universit\"at Mannheim\newline
D--68131 Mann\-heim, Germany}
\email{potthoff@math.uni-mannheim.de}

\author[R.~Schrader]{Robert Schrader}
\address{Robert Schrader\newline
Institut f\"{u}r Theoretische Physik\newline
Freie Universit\"{a}t Berlin, Arnimallee~14\newline
D--14195 Berlin, Germany}
\email{schrader@physik.fu-berlin.de}

\copyrightinfo{2011}{V.~Kostrykin, J.~Potthoff, R.~Schrader}

\begin{document}
\begin{abstract}
Various equivalent conditions for a semigroup or a resolvent generated
by a Markov process to be of Feller type are given.
\end{abstract}

\maketitle
\thispagestyle{empty}

The Feller property of the semigroup generated by a Markov process plays
a prominent role in the theory of stochastic processes. This is mainly due to
the fact that if the Feller property holds true, then --- under the additional
assumption of right continuity of the paths --- the simple Markov property implies the
strong Markov property (e.g., \cite[Theorem~III.3.1]{ReYo91} or
\cite[Theorem~III.15.3]{Wi79}).

However, in many instances it is of advantage to consider the associated resolvent
instead of the semigroup. Therefore we present in this note a result
which states various forms of the equivalence of the Feller property as expressed
in terms of the semigroup or of the resolvent. The material here seems to be quite
well-known, and our presentation of it owes very much to~\cite{Ra56} --- most notably
the inversion formula for the Laplace transform, equation~\eqref{inv_L} in connection with
lemma~\ref{lem_5}. On the other hand, we were not able to locate a reference
where the results are collected and stated in the form of the theorem given below.

Assume that $(E,d)$ is a locally compact separable metric space with Borel
$\gs$--algebra denoted by $\cB(E)$. $B(E)$ denotes the space of bounded measurable
real valued functions on $E$, $\Coe$ the subspace of continuous functions vanishing
at infinity. $B(E)$ and $\Coe$ are equipped with the sup-norm $\|\,\cdot\,\|$.

The following definition is as in~\cite{ReYo91}:
\begin{definition}  \label{def_1}
    A \emph{Feller semigroup} is a family $U=(U_t,\,t\ge 0)$ of positive
    linear operators on $\Coe$ such that
    \begin{enum_a}
        \item $U_0=\text{id}$ and $\|U_t\|\le 1$ for every $t\ge 0$;
        \item $U_{t+s} = U_t\comp U_s$ for every pair $s$, $t\ge 0$;
        \item $\lim_{t\downarrow 0} \|U_t f - f\|=0$ for every $f\in \Coe$.
    \end{enum_a}
\end{definition}

Analogously we define
\begin{definition}  \label{def_2}
    A \emph{Feller resolvent} is a family $R=(R_\gl,\,\gl>0)$ of positive
    linear operators on $\Coe$ such that
    \begin{enum_a}
        \item $\|R_\gl\|\le \gl^{-1}$ for every $\gl>0$;
        \item $R_\gl - R_\mu = (\mu-\gl) R_\gl\comp R_\mu$ for every
                pair $\gl$, $\mu>0$;
        \item $\lim_{\gl\to\infty}\|\gl R_\gl f - f\|=0$ for every $f\in\Coe$.
    \end{enum_a}
\end{definition}

In the sequel we shall focus our attention on semigroups $U$ and resolvents $R$
associated with an $E$--valued Markov process, and which are \emph{a priori} defined
on $B(E)$. (In our notation, we shall not distinguish between $U$ and $R$ as defined
on $B(E)$ and their restrictions to $\Coe$.)

Let $X = (X_t,\,t\ge 0)$ be a Markov process with state space $E$, and let
$(P_x,\,x\in E)$ denote the associated family of probability measures on some measurable
space $(\gO,\cA)$, so that in particular $P_x(X_0=x) = 1$. $E_x(\,\cdot\,)$ denotes the
expectation with respect to $P_x$. We assume throughout that for every $f\in B(E)$
the mapping
\begin{equation*}
    (t,x) \mapsto E_x\bigl(f(X_t)\bigr)
\end{equation*}
is measurable from $\R_+\times E$ into $\R$. The semigroup $U$ and resolvent $R$
associated with $X$ act on $B(E)$ as follows. For $f\in B(E)$, $x\in E$, $t\ge 0$,
and $\gl>0$ set
\begin{align}
    U_t f(x)   &= E_x\bigl(f(X_t)\bigr),  \label{eq_1}\\
    R_\gl f(x) &= \int_0^\infty e^{-\gl t} U_t f (x)\,dt.  \label{eq_2}
\end{align}
Property~(a) of Definitions~\ref{def_1} and \ref{def_2} is obviously satisfied. The
semigroup property, (b) in definition~\ref{def_1}, follows from the Markov property
of $X$, and this in turn implies the resolvent equation, (b) of
definition~\ref{def_2}. Moreover, it follows also from the Markov property of $X$
that the semigroup and the resolvent commute. On the other hand, in general neither
the property that $U$ or $R$ map $\Coe$ into itself, nor the strong continuity
property (c) in Definitions~\ref{def_1}, \ref{def_2} hold true on $B(E)$ or on
$\Coe$.

If $W$ is a subspace of $B(E)$ the resolvent equation shows that the image of $W$
under $R_\gl$ is independent of the choice of $\gl>0$, and in the sequel we shall
denote the image by $RW$. Furthermore, for simplicity we shall write $U\Coe\subset
\Coe$, if $U_t f\in\Coe$ for all $t\ge 0$, $f\in\Coe$.

\begin{theorem*} \label{thm}
The following statements are equivalent:
\begin{enum_a}
    \item $U$ is Feller.
    \item $R$ is Feller.
    \item $U\Coe\subset\Coe$, and for all $f\in\Coe$, $x\in E$,
            $\lim_{t\downarrow 0} U_t f(x) = f(x)$.
    \item $U\Coe\subset\Coe$, and for all $f\in\Coe$, $x\in E$,
            $\lim_{\gl\rightarrow \infty} \gl R_\gl f(x) = f(x)$.
    \item $R\Coe\subset\Coe$, and for all $f\in\Coe$, $x\in E$,
            $\lim_{t\downarrow 0} U_t f(x) = f(x)$.
    \item $R\Coe\subset\Coe$, and for all $f\in\Coe$, $x\in E$,
            $\lim_{\gl\rightarrow \infty} \gl R_\gl f(x) = f(x)$.
\end{enum_a}
\end{theorem*}

We prepare a sequence of lemmas. The first one follows directly from the
dominated convergence theorem:

\begin{lemma}   \label{lem_3}
Assume that for $f\in B(E)$, $U_t f \rightarrow f$ as $t\downarrow 0$. Then $\gl
R_\gl f \rightarrow f$ as $\gl\to+\infty$.
\end{lemma}

\begin{lemma}   \label{lem_4}
The semigroup $U$ is strongly continuous on $RB(E)$.
\end{lemma}

\begin{proof}
If strong continuity at $t=0$ has been shown, strong continuity at $t>0$ follows
from the semigroup property of $U$, and the fact that $U$ and $R$ commute. Therefore
it is enough to show strong continuity at $t=0$.

Let $f\in B(E)$, $\gl>0$, $t>0$, and consider for $x\in E$ the following computation
\begin{align*}
    U_t R_\gl f(x) &- R_\gl f(x)\\
        &= \int_0^\infty e^{-\gl s} E_x\bigl(f(X_{t+s})\bigr)\,ds
                - \int_0^\infty e^{-\gl s} E_x\bigl(f(X_s)\bigr)\,ds\\
        &= e^{\gl t} \int_t^\infty e^{-\gl s} E_x\bigl(f(X_s)\bigr)\,ds
                - \int_0^\infty e^{-\gl s} E_x\bigl(f(X_s)\bigr)\,ds\\
        &= \bigl(e^{\gl t}-1\bigr) \int_t^\infty e^{-\gl s} E_x\bigl(f(X_s)\bigr)\,ds
                - \int_0^t e^{-\gl s} E_x\bigl(f(X_s)\bigr)\,ds\\
\end{align*}
where we used Fubini's theorem and the Markov property of $X$. Thus we get the
following estimation
\begin{align*}
    \bigl\|U_t R_\gl f - R_\gl f\bigr\|
        &\le \biggl(\bigl(e^{\gl t} - 1\bigr)\int_t^\infty e^{-\gl s}\,ds
                +\int_0^t e^{-\gl s}\,ds\biggl)\, \|f\|\\
        &= \frac{2}{\gl}\, \bigl(1-e^{-\gl t}\bigr)\,\|f\|,
\end{align*}
which converges to zero as $t$ decreases to zero.
\end{proof}

For $\gl>0$, $t\ge 0$, $f\in B(E)$, $x\in E$ set
\begin{equation}    \label{inv_L}
    U^\gl_t f(x) = \sum_{n=1}^\infty \frac{(-1)^{n+1}}{n!}
                        \,n\gl\, e^{n\gl t}\, R_{n\gl} f(x).
\end{equation}
Observe that, because of $n\gl\|R_{n\gl} f\| \le \|f\|$, the last sum converges in
$B(E)$.

For the proof of the next lemma we refer the reader
to~\cite[p.~477~f]{Ra56}:

\begin{lemma}   \label{lem_5}
For all $t\ge 0$, $f\in RB(E)$, $U^\gl_t f$ converges in $B(E)$ to $U_t f$ as
$\gl$ tends to infinity.
\end{lemma}

\begin{lemma}   \label{lem_6}
If $U_t\Coe \subset \Coe$ for all $t\ge 0$, then $R_\gl\Coe\subset \Coe$, for all $\gl>0$.
If $R_\gl\Coe\subset \Coe$, for some $\gl>0$, and $R_\gl\Coe$ is dense in $\Coe$, then
$U_t\Coe \subset \Coe$ for all $t\ge 0$.
\end{lemma}

\begin{proof}
Assume that  $U_t\Coe \subset \Coe$ for all $t\ge 0$, let $f\in\Coe$, $x\in E$, and
suppose that $(x_n,\,n\in\N)$ is a sequence converging in $(E,d)$ to $x$. Then a
straightforward application of the dominated convergence theorem shows that for
every $\gl>0$, $R_\gl f(x_n)$ converges to $R_\gl f(x)$. Hence $R_\gl f\in \Coe$.

Now assume that that $R_\gl\Coe\subset \Coe$, for some and therefore for all
$\gl>0$, and that $R_\gl\Coe$ is dense in $\Coe$. Consider $f\in R\Coe$, $t>0$, and
for $\gl>0$ define $U^\gl_t f$ as in equation~\eqref{inv_L}. Because
$R_{n\gl}f\in\Coe$ and the series in formula~\eqref{inv_L} converges uniformly in
$x\in E$, we get $U^\gl_t f\in\Coe$. By lemma~\ref{lem_5}, we find that $U^\gl_t
f$ converges uniformly to $U_t f$ as $\gl\to+\infty$. Hence $U_t f\in\Coe$. Since
$R\Coe$ is dense in $\Coe$, $U_t$ is a contraction and $\Coe$ is closed, we get that
$U_t\Coe\subset\Coe$ for every $t\ge 0$.
\end{proof}

The following lemma is proved as a part of Theorem~17.4 in~\cite{Ka97}
(cf.\ also the proof of Proposition~2.4 in~\cite{ReYo91}).

\begin{lemma}   \label{lem_7}
Assume that $R\Coe\subset \Coe$, and that for all $x\in E$, $f\in\Coe$,
$\lim_{\gl\to\infty} \gl R_\gl f(x) = f(x)$. Then $R\Coe$ is dense in $\Coe$.
\end{lemma}

If for all $f\in\Coe$, $x\in E$, $U_t f(x)$ converges to $f(x)$ as $t$ decreases to
zero, then similarly as in the proof of lemma~\ref{lem_3} we get that $\gl R_\gl
f(x)$ converges to $f(x)$ as $\gl\to+\infty$. Thus we obtain the following

\begin{corollary}   \label{cor_8}
Assume that $R\Coe\subset \Coe$, and that for all $x\in E$, $f\in\Coe$,
$\lim_{t\downarrow 0} U_t f(x) = f(x)$. Then $R\Coe$ is dense in $\Coe$.
\end{corollary}

Now we can come to the

\begin{proof}[Proof of the theorem]
We begin by proving the equivalence of statements~(a), (b), (d), and~(f):

\vspace{.5\baselineskip}\noindent
``(a)\Ra(b)'' Assume that $U$ is Feller. From lemma~\ref{lem_6} it follows that
$R_\gl\Coe\subset\Coe$, $\gl>0$. Let $f\in\Coe$. Since $U$ is strongly continuous on
$\Coe$, lemma~\ref{lem_3} implies that $\gl R_\gl f$ converges to $f$ as $\gl$ tends
to $+\infty$. Hence $R$ is Feller.

\vspace{.25\baselineskip}\noindent
``(b)\Ra(f)'' This is trivial.

\vspace{.25\baselineskip}\noindent
``(f)\Ra(d)'' By lemma~\ref{lem_7}, $R\Coe$ is dense in $\Coe$, and therefore
lemma~\ref{lem_6} entails that $U\Coe\subset\Coe$.

\vspace{.25\baselineskip}\noindent
``(d)\Ra(a)'' By lemmas~\ref{lem_6} and \ref{lem_7},  $R\Coe$ is dense in
$\Coe$, and therefore by lemma~\ref{lem_4}, $U$ is strongly continuous on $\Coe$. Thus
$U$ is Feller.

\vspace{.25\baselineskip}
Now we prove the equivalence of~(a), (c), and (e):

\vspace{.25\baselineskip}\noindent
``(a)\Ra(c)'' This is trivial.

\vspace{.25\baselineskip}\noindent
``(c)\Ra(e)'' This follows directly from Lemma~\ref{lem_6}.

\vspace{.25\baselineskip}\noindent
``(e)\Ra(a)'' By corollary~\ref{cor_8}, $R\Coe$ is dense in $\Coe$, hence it follows
from lem\-ma~\ref{lem_6} that $U\Coe\subset\Coe$. Furthermore, lemma~\ref{lem_4}
implies the strong continuity of $U$ on $R\Coe$, and by density therefore on $\Coe$.
(a) follows.
\end{proof}

\vspace{.5\baselineskip}\noindent
\textbf{Acknowledgement.} J.~P. thanks Florian Werner for a useful discussion.

\providecommand{\bysame}{\leavevmode\hbox to3em{\hrulefill}\thinspace}
\providecommand{\MR}{\relax\ifhmode\unskip\space\fi MR }
\providecommand{\MRhref}[2]{%
  \href{http://www.ams.org/mathscinet-getitem?mr=#1}{#2}
}
\providecommand{\href}[2]{#2}

\end{document}